\documentclass[12pt,reqno]{amsart}
\usepackage{amscd,amssymb,amsmath,multicol,float,scalefnt,cancel}
\restylefloat{table}
\newtheorem{thm}[equation]{Theorem}
\numberwithin{equation}{section}
\newtheorem{cor}[equation]{Corollary}

\newtheorem{lem}[equation]{Lemma}

\newtheorem{defin}[equation]{Definition}

\newtheorem{prop}[equation]{Proposition}

\newtheorem{tab}[equation]{Table}

\begin{document}
\raggedbottom \voffset=-.7truein \hoffset=0truein \vsize=8truein
\hsize=6truein \textheight=8truein \textwidth=6truein
\baselineskip=18truept

\def\mapright#1{\ \smash{\mathop{\longrightarrow}\limits^{#1}}\ }
\def\mapleft#1{\smash{\mathop{\longleftarrow}\limits^{#1}}}
\def\mapup#1{\Big\uparrow\rlap{$\vcenter {\hbox {$#1$}}$}}
\def\mapdown#1{\Big\downarrow\rlap{$\vcenter {\hbox {$\ssize{#1}$}}$}}
\def\mapne#1{\nearrow\rlap{$\vcenter {\hbox {$#1$}}$}}
\def\mapse#1{\searrow\rlap{$\vcenter {\hbox {$\ssize{#1}$}}$}}
\def\mapr#1{\smash{\mathop{\rightarrow}\limits^{#1}}}
\def\ss{\smallskip}
\def\vp{v_1^{-1}\pi}
\def\at{{\widetilde\alpha}}
\def\sm{\wedge}
\def\la{\langle}
\def\ra{\rangle}
\def\on{\operatorname}
\def\ol#1{\overline{#1}{}}
\def\spin{\on{Spin}}
\def\lbar{\ell}
\def\qed{\quad\rule{8pt}{8pt}\bigskip}
\def\ssize{\scriptstyle}
\def\a{\alpha}
\def\bz{{\Bbb Z}}
\def\im{\on{im}}
\def\ct{\widetilde{C}}
\def\ext{\on{Ext}}
\def\sq{\on{Sq}}
\def\eps{\epsilon}
\def\ar#1{\stackrel {#1}{\rightarrow}}
\def\br{{\bold R}}
\def\bC{{\bold C}}
\def\bA{{\bold A}}
\def\bB{{\bold B}}
\def\bD{{\bold D}}
\def\bh{{\bold H}}
\def\bQ{{\bold Q}}
\def\bP{{\bold P}}
\def\bx{{\bold x}}
\def\bo{{\bold{bo}}}
\def\si{\sigma}
\def\Vbar{{\overline V}}
\def\dbar{{\overline d}}
\def\wbar{{\overline w}}
\def\Sum{\sum}
\def\tfrac{\textstyle\frac}
\def\tb{\textstyle\binom}
\def\Si{\Sigma}
\def\w{\wedge}
\def\equ{\begin{equation}}
\def\b{\beta}
\def\G{\Gamma}
\def\g{\gamma}
\def\k{\kappa}
\def\psit{\widetilde{\Psi}}
\def\tht{\widetilde{\Theta}}
\def\psiu{{\underline{\Psi}}}
\def\thu{{\underline{\Theta}}}
\def\aee{A_{\text{ee}}}
\def\aeo{A_{\text{eo}}}
\def\aoo{A_{\text{oo}}}
\def\aoe{A_{\text{oe}}}
\def\vbar{{\overline v}}
\def\endeq{\end{equation}}
\def\sn{S^{2n+1}}
\def\zp{\bold Z_p}
\def\A{{\cal A}}
\def\P{{\mathcal P}}
\def\cQ{{\mathcal Q}}
\def\cj{{\cal J}}
\def\zt{{\bold Z}_2}
\def\bs{{\bold s}}
\def\bof{{\bold f}}
\def\bq{{\bold Q}}
\def\be{{\bold e}}
\def\Hom{\on{Hom}}
\def\ker{\on{ker}}
\def\coker{\on{coker}}
\def\da{\downarrow}
\def\colim{\operatornamewithlimits{colim}}
\def\zphat{\bz_2^\wedge}
\def\io{\iota}
\def\Om{\Omega}
\def\Prod{\prod}
\def\e{{\cal E}}
\def\exp{\on{exp}}
\def\abar{{\overline a}}
\def\xbar{{\overline x}}
\def\ybar{{\overline y}}
\def\zbar{{\overline z}}
\def\Rbar{{\overline R}}
\def\nbar{{\overline n}}
\def\cbar{{\overline c}}
\def\qbar{{\overline q}}
\def\bbar{{\overline b}}
\def\et{{\widetilde E}}
\def\ni{\noindent}
\def\coef{\on{coef}}
\def\den{\on{den}}
\def\lcm{\on{l.c.m.}}
\def\vi{v_1^{-1}}
\def\ot{\otimes}
\def\psibar{{\overline\psi}}
\def\mhat{{\hat m}}
\def\exc{\on{exc}}
\def\ms{\medskip}
\def\ehat{{\hat e}}
\def\etao{{\eta_{\text{od}}}}
\def\etae{{\eta_{\text{ev}}}}
\def\dirlim{\operatornamewithlimits{dirlim}}
\def\gt{\widetilde{L}}
\def\lt{\widetilde{\lambda}}
\def\st{\widetilde{s}}
\def\ft{\widetilde{f}}
\def\sgd{\on{sgd}}
\def\lfl{\lfloor}
\def\rfl{\rfloor}
\def\ord{\on{ord}}
\def\gd{{\on{gd}}}
\def\rk{{{\on{rk}}_2}}
\def\nbar{{\overline{n}}}
\def\lg{{\on{lg}}}
\def\cB{\mathcal{B}}
\def\cS{\mathcal{S}}
\def\cP{\mathcal{P}}
\def\N{{\Bbb N}}
\def\Z{{\Bbb Z}}
\def\Q{{\Bbb Q}}
\def\R{{\Bbb R}}
\def\C{{\Bbb C}}
\def\l{\left}
\def\r{\right}
\def\mo{\on{mod}}
\def\xt{\times}
\def\notimm{\not\subseteq}
\def\Remark{\noindent{\it  Remark}}

\def\*#1{\mathbf{#1}}
\def\0{$\*0$}
\def\1{$\*1$}
\def\22{$(\*2,\*2)$}
\def\33{$(\*3,\*3)$}
\def\ss{\smallskip}
\def\ssum{\sum\limits}
\def\dsum{\displaystyle\sum}
\def\la{\langle}
\def\ra{\rangle}
\def\on{\operatorname}
\def\o{\on{o}}
\def\U{\on{U}}
\def\lg{\on{lg}}
\def\a{\alpha}
\def\bz{{\Bbb Z}}
\def\eps{\varepsilon}
\def\bc{{\bold C}}
\def\bN{{\bold N}}
\def\nut{\widetilde{\nu}}
\def\tfrac{\textstyle\frac}
\def\b{\beta}
\def\G{\Gamma}
\def\g{\gamma}
\def\zt{{\Bbb Z}_2}
\def\zth{{\bold Z}_2^\wedge}
\def\bs{{\bold s}}
\def\bx{{\bold x}}
\def\bof{{\bold f}}
\def\bq{{\bold Q}}
\def\be{{\bold e}}
\def\lline{\rule{.6in}{.6pt}}
\def\xb{{\overline x}}
\def\xbar{{\overline x}}
\def\ybar{{\overline y}}
\def\zbar{{\overline z}}
\def\ebar{{\overline \be}}
\def\nbar{{\overline n}}
\def\rbar{{\overline r}}
\def\Mbar{{\overline M}}
\def\et{{\widetilde e}}
\def\ni{\noindent}
\def\ms{\medskip}
\def\ehat{{\hat e}}
\def\what{{\widehat w}}
\def\Yhat{{\widehat Y}}
\def\nbar{{\overline{n}}}
\def\minp{\min\nolimits'}
\def\mul{\on{mul}}
\def\N{{\Bbb N}}
\def\Z{{\Bbb Z}}
\def\Q{{\Bbb Q}}
\def\R{{\Bbb R}}
\def\C{{\Bbb C}}
\def\notint{\cancel\cap}
\def\cS{\mathcal S}
\def\cR{\mathcal R}
\def\cC{\mathcal C}
\def\cP{\mathcal P}
\def\el{\ell}
\def\TC{\on{TC}}
\def\dstyle{\displaystyle}
\def\ds{\dstyle}
\def\Wbar{\wbar}
\def\zcl{\on{zcl}}
\def\Vb#1{{\overline{V_{#1}}}}

\def\Remark{\noindent{\it  Remark}}
\title[Topological complexity of the space of planar $n$-gons]
{Topological complexity (within 1) of the space of isometry classes of planar $n$-gons for sufficiently large $n$}
\author{Donald M. Davis}
\address{Department of Mathematics, Lehigh University\\Bethlehem, PA 18015, USA}
\email{dmd1@lehigh.edu}
\date{September 7, 2016}


\maketitle
\begin{abstract}  Hausmann and Rodriguez classified spaces of planar $n$-gons mod isometry according to their genetic code, which is a collection of sets (called genes) containing $n$. Omitting the $n$ yields what we call gees. We prove that, for a set of gees with largest gee of size $k>0$, the topological complexity (TC) of the associated space of $n$-gons is either $2n-5$ or $2n-6$ if $n\ge 2k+3$. We present evidence that suggests that it is very rare that the TC is not equal to $2n-5$ or $2n-6$.
 \end{abstract}
\section{Introduction}\label{intro}
The topological complexity, $\TC(X)$, of a topological space $X$ is, roughly, the number of rules required to specify how to move between any two points of $X$. A ``rule'' must be such that the choice of path varies continuously with the choice of endpoints. (See \cite[\S4]{F}.) We continue our study, begun in \cite{Forum}, of $\TC(X)$ where $X=\Mbar(\lbar)$ is the space of equivalence classes of oriented $n$-gons in  the plane with consecutive sides of lengths $\ell_1,\ldots,\ell_n$, identified under translation, rotation, and reflection. (See, e.g., \cite[\S6]{HK}.) Here $\ell=(\ell_1,\ldots,\ell_n)$ is an $n$-tuple of positive real numbers. Thus
 $$\Mbar(\ell)=\{(z_1,\ldots,z_n)\in (S^1)^n:\ell_1z_1+\cdots+\ell_nz_n=0\}/O(2).$$
We can think of the sides of the polygon as linked arms of a robot, and then $\TC(X)$ is the number of rules required to program the robot to move from any configuration to any other.

We assume that  $\lbar$ is {\it generic}, which means that there is no subset $S\subset[\![n]\!]$ with $\ds\sum_{i\in S}\ell_i=\dsum_{i\not\in S}\ell_i$.
Here $[\![n]\!]=\{1,\ldots,n\}$, notation that will be used throughout the paper.
When $\lbar$ is generic, $\Mbar(\ell)$ is a connected
$(n-3)$-manifold (\cite[p.314]{HK}), and hence, by \cite[Cor 4.15]{F}, satisfies
\begin{equation}\label{ineq1}\TC(\Mbar(\lbar))\le 2n-5.\end{equation}

The mod-2 cohomology ring $H^*(\Mbar(\ell))$ was determined in \cite{HK}. See Theorem \ref{mycoh} for our interpretation. All of our cohomology groups have coefficients in $\zt$, omitted from the notation. We shall prove that for most length-$n$  vectors $\ell$, there is a nonzero product in $H^*(\Mbar(\ell)\times \Mbar(\ell))$ of $2n-7$ classes of the form
$z\otimes1+1\otimes z$, which implies $\TC(\Mbar(\ell))\ge 2n-6$ by \cite[Cor 4.40]{F},  within 1 of being optimal by (\ref{ineq1}). We say that this lower bound for $\TC(\Mbar(\ell))$ is obtained by zcl (zero-divisor cup length) consideration, or that $\zcl(\Mbar(\ell))\ge 2n-7$. We write $\zbar=z\otimes1+1\otimes z$.

To formulate our result, we recall the concepts of genetic code and gees.
Since permuting the $\ell_i$'s does not affect the space up to homeomorphism, we may assume $\ell_1\le\cdots\le\ell_n$.  We also assume that $\ell_n<\ell_1+\cdots+\ell_{n-1}$, so that $\Mbar(\ell)$ is nonempty.
It is well-understood (e.g., \cite[\S2]{HK})  that the homeomorphism type of $\Mbar(\lbar)$ is determined by which subsets $S$ of $[\![n]\!]$ are {\it short}, which means that
$\dsum_{i\in S}\ell_i<\dsum_{i\not\in S}\ell_i$. Define a partial order on the power set of $[\![n]\!]$  by $S\le T$ if  $S=\{s_1,\ldots,s_\ell\}$ and $T\supset\{t_1,\ldots,t_\ell\}$ with $s_i\le t_i$ for all $i$.
As introduced in \cite{HR}, the {\it genetic code} of $\lbar$  is the  set of maximal elements (called {\it genes}) in the set of short subsets of $[\![n]\!]$ which contain $n$.
The homeomorphism type of $\Mbar(\lbar)$ is determined by the genetic code of $\lbar$. A list of all genetic codes for $n\le 9$ appears in \cite{web}. For $n=6$, 7, and 8, there are 20, 134, and 2469 genetic codes, respectively.

In \cite{Forum}, we introduced the term ``gee,'' to refer to a gene without listing the $n$. Also, a {\it subgee} is any set which is $\le G$ for some gee $G$, under the partial ordering just described. Thus the subgees are just the sets $S\subset[\![n-1]\!]$ for which $S\cup \{ n\}$ is short. Our main theorem is as follows. Let $\lg(-)=\lfloor\log_2(-)\rfloor$.
\begin{thm} \label{mainthm} For a set of gees with largest gee of size $k>0$, the associated space of $n$-gons $\Mbar(\ell)$ satisfies \begin{equation}\label{ineqs}2n-6\le\TC(\Mbar(\ell))\le 2n-5\end{equation}  if $n\ge k+2^{\lg(k)}+3$.\end{thm}
\ni Here we mean that $\ell$ is a length-$n$  vector whose genetic code has the given set as its gees, with $n$ appended to form its genes.

The stipulation  $k>0$ excludes the $n$-gon space whose genetic code is $\la\{n\}\ra$. One length vector with this genetic code is $(1^{n-1},n-2)$. A polygon space with this genetic code is homeomorphic to real projective space $RP^{n-3}$, and it is known that, except for  $RP^1$, $RP^3$, and $RP^7$, $\TC(RP^{n-3})$ equals the immersion dimension plus 1, which is usually much less than $2n-6$.(\cite{FTY})

We will prove two results, Theorems \ref{str1} and \ref{str2}, which suggest that (\ref{ineqs}) fails only very rarely. As noted above, it usually fails if the genetic code is $\la\{n\}\ra$. At the other extreme, it fails if the genetic code is $\la\{n,n-3,n-4,\ldots,1\}\ra$, in which case $\Mbar(\ell)$ is a torus $T^{n-3}$ of topological complexity $n-2$.
We showed in \cite{Forum} that, of the 132 equivalence classes of 7-gon spaces excluding $RP^4$ and $T^4$, there are only two, namely those with a single gene, $7321$ or $7521$, for which we cannot prove that they satisfy (\ref{ineqs}). Here we have begun the usual practice of writing genes or gees consisting of single-digit numbers by concatenating. In Theorem \ref{str1}, we obtain a similar result for 8-gon spaces, but this time with just two exceptions out of 2467 equivalence classes.
\begin{thm}\label{str1} Excluding $RP^5$ and $T^5$, the only spaces $\Mbar(\ell)$ with $n=8$ which might not satisfy (\ref{ineqs}) are those with a single gene, $84321$ or $86321$.\end{thm}

In Section \ref{sec2}, we specialize to monogenic codes and prove in Theorem \ref{str2} that the only genetic codes with a single gene of size 5 (for {\em any} $n$) which do not necessarily satisfy (\ref{ineqs}) are those noted above.

Theorem \ref{mainthm} is an immediate consequence of Theorems \ref{genlthm} and \ref{nonzthm}. We introduce those theorems by giving
 our interpretation (\cite[Thm 2.1]{Forum}) of \cite[Cor 9.2]{HK}, the complete structure of the mod 2 cohomology ring of $\Mbar(\ell)$.
\begin{thm}\label{mycoh}  If $\ell$ has length $n$, the ring $H^*(\Mbar(\ell))$ is generated by classes $$R,V_1,\ldots,V_{n-1}\in H^1(\Mbar(\ell))$$ subject to only the following relations:
\begin{enumerate}
\item All monomials of the same degree which are divisible by the same $V_i$'s are equal. Hence, letting $V_S:=\prod_{i\in S}V_i$,  monomials $R^{d-|S|}V_S$ for $S\subset[\![n-1]\!]$ span $H^d(\Mbar(\ell))$.
\item $V_S=0$ unless $S$ is a subgee of $\ell$.
\item For every subgee $S$ with $|S|\ge n-2-d$, there is a relation $\cR_S$ in $H^d(\Mbar(\ell))$, which says
\begin{equation}\label{rel}\sum_{T\notint S}R^{d-|T|}V_T=0.\end{equation}
\end{enumerate}
\end{thm}

 It is convenient to let $m=n-3$, which we do throughout. Note that $\Mbar(\ell)$ is an $m$-manifold. The proof of Theorem \ref{mainthm} is split into two cases depending upon whether or not $R^m=0$.
 \begin{thm}\label{genlthm} Suppose $R^m=0$. Then there exists a positive integer $r$ and distinct integers $i_1,\ldots,i_r$ such that $R^{m-r}V_{i_1}\cdots V_{i_r}\ne0\in H^m(\Mbar(\ell))\approx\zt$, but for all proper subsets $S\subsetneq\{i_1,\ldots,i_r\}$, $\ds R^{m-|S|}\prod_{i\in S}V_i=0$. Assume $m\ge r+2^{\lg r}$, and let $f=\lg(m-r)\ge \lg(r)$. Let $A=2m-2r-2^{f+1}+3.$ Then,
$$\prod_{j=1}^{r-1}{\Vb{i_j}}^3\cdot {\Vb{i_r}}^A\cdot \Rbar^{2m+2-A-3r}\ne0\in H^{2m-1}(\Mbar(\ell)\times \Mbar(\ell)),$$
and hence $\TC(\Mbar(\ell))\ge 2m=2n-6$.
\end{thm}

\begin{thm}\label{nonzthm} Suppose $R^m\ne0$. If $m$ is a 2-power, then $\Rbar^{2m-1}\ne0$. If $m$ is not a 2-power, then there exist  positive integers $t$ and $A$ and distinct integers $i_1,\ldots,i_{t+1}$ such that
$$\Vb{i_1}\cdots\Vb{i_t}\,\Vb{i_{t+1}}^A\,\Rbar^{2m-A-t-1}\ne0\in H^{2m-1}(\Mbar(\ell)\times \Mbar(\ell)).$$
Hence in either case $\TC(\Mbar(\ell))\ge 2m=2n-6$.
\end{thm}
\ni In Theorem \ref{nonzthm}, any $m$ large enough, with respect to the given gees, to yield a valid genetic code works in the theorem.

In Section \ref{pfsec}, we prove Theorems \ref{genlthm}, \ref{nonzthm}, and \ref{str1}.
In Section \ref{Incrsec}, we discuss the effect on the length vectors and the cohomology ring of increasing $n$ while leaving the gees unchanged, and give some examples regarding the sharpness of the bound on how large $m$ must be in Theorem \ref{genlthm}.
In Section \ref{sec2},  we give several explicit families of gees of arbitrarily large size to which Theorem \ref{genlthm} applies. We also prove in Theorem \ref{str2}  that there are only three size-5 genes for which we cannot prove (\ref{ineqs}).

\section{Proofs of Theorems \ref{genlthm}, \ref{nonzthm}, and \ref{str1}}\label{pfsec}
In this section, which we feel is the heart of the paper, we prove Theorems \ref{genlthm}, \ref{nonzthm}, and \ref{str1}.

\begin{proof}[Proof of Theorem \ref{genlthm}] We begin with the simple observation that if $R^m=0$, then $r$ can be chosen as
$$r=\min\{t:\ \exists \text{ distinct }i_1,\ldots,i_t\text{ with }R^{m-t}V_{i_1}\cdots V_{i_t}\ne 0\}.$$

 First observe that $A\ge3$ and the exponent of $\Rbar$ is $2^{f+1}-1-r\ge0$, since $f\ge\lg(r)$. By minimality of $r$, in the expansion of the product, factors $V_{i_j}^3\otimes1$, $1\otimes V_{i_j}^3$, $V_{i_r}^A\otimes 1$, and $1\otimes V_{i_r}^A$ will yield 0 in products. A product of $s$ of the $V_{i_j}^2$'s and $(r-1-s)$ of the $V_{i_j}$'s can be written as $R^sP$, where $P:=\ds\prod_{j=1}^{r-1}V_{i_j}$.
Thus our product expands in bidegree $(m,m-1)$ as
\begin{equation}\label{sum}\sum_{s=0}^{r-1}\tbinom{r-1}sR^sP\sum_{j=1}^{A-1}\tbinom AjV_{i_r}^j\tbinom{2m+2-A-3r}{m-s-r-j+1}R^{m-s-r-j+1}\otimes R^{r-1-s}PV_{i_r}^{A-j}R^{m+1-A-2r+s+j}.\end{equation}

Let $\phi:H^m(\Mbar(\ell))\to\zt$ be the Poincar\'e duality isomorphism. Let $V_I$ denote any product of distinct classes $V_i$. There is a homomorphism $\psi:H^{m-1}(\Mbar(\ell))\to\zt$ satisfying that $\psi(R^{m-1-|I|}V_I)\ne0$ iff $\phi(R^{m-|I|}V_I)\ne0$. This follows from Theorem \ref{mycoh} since the relations in $H^{m-1}(\Mbar(\ell))$ are also relations in $H^m(\Mbar(\ell))$ (with a dimension shift). Thus $\phi\ot\psi$ applied to any summand of (\ref{sum}) which has a nonzero coefficient, mod 2, is nonzero, and so $\phi\ot\psi$ applied to (\ref{sum}) equals
\begin{eqnarray*}&&\sum_{s=0}^{r-1}\sum_{j=1}^{A-1}\tbinom{r-1}s\tbinom{A}j\tbinom{2m+2-A-3r}{m-s-r-j+1}\\
&=&\sum_{s=0}^{r-1}\sum_{j=0}^{A}\tbinom{r-1}s\tbinom{A}j\tbinom{2m+2-A-3r}{m-s-r-j+1}+\sum_{s=0}^{r-1}\tbinom{r-1}s\tbinom{2m+2-A-3r}{m-s-r+1}+\sum_{s=0}^{r-1}\tbinom{r-1}s\tbinom{2m+2-A-3r}{m-s-r-A+1}.
\end{eqnarray*}
By Vandermonde's identity, this equals
\begin{equation}\tbinom{2m-2r+1}{m-r+1}+\tbinom{2m-2r+1-A}{m-r+1}+\tbinom{2m-2r+1-A}{m-r+1-A}.\label{image}\end{equation}
The last binomial coefficient equals $\tbinom{2m-2r+1-A}{m-r}$, and so the sum of the last two equals $\tbinom{2m-2r+2-A}{m-r+1}$. Inserting now the value of $A$, we find that the image of our class equals
$$\tbinom{2m-2r+1}{m-r+1}+\tbinom{2^{f+1}-1}{m-r+1}.$$
For $2^f\le m-r\le 2^{f+1}-1$, this expression equals 1, coming from the first term if $m-r=2^{f+1}-1$ and from the second term otherwise.
\end{proof}

The proof of Theorem \ref{nonzthm} is a bit more elaborate. We will always be using the homomorphism $\psi:H^{m-1}(\Mbar(\ell))\to\zt$
which equals the Poincar\'e duality isomorphism $\phi:H^m(\Mbar(\ell))\to\zt$ in the sense of the preceding proof.
We first observe that if $\phi(R^m)\ne0$ and $\psi(R^{m-1})\ne0$ and $m=2^e$, then Theorem \ref{nonzthm} is true since $$(\phi\ot\psi)(\Rbar^{2m-1})=\tbinom{2m-1}m\phi(R^{m})\psi(R^{m-1})\ne0.$$
In the rest of this section, we assume $m$ is not a 2-power.

The following key lemma rules out certain possibilities for $\phi$.
\begin{lem}\label{ref} It cannot happen that there is a subset $S\subset[\![m+2]\!]$ such that
$$\phi(R^{m-|I|}V_I)=\begin{cases}1&I\subset S\\
0&\text{otherwise.}\end{cases}$$
\end{lem}
\begin{proof} We assume such a set $S$ exists and will derive a contradiction.
Let $k$ denote the size of the largest subgee. By \cite[Cor 1.6]{D}, $\phi(R^{m-k}V_{i_1}\cdots V_{i_k})=1$ whenever $\{i_1,\ldots,i_k\}$ is a subgee. Although the result in \cite{D} is  apparently only referring to monogenic codes, the proof applies more generally. By the assumption, we conclude that there can only be one subgee of size $k$, and it must be $[\![k]\!]$. The sum $\cR_{[\![k]\!]}+\cR_{[\![k-1]\!]}$ of relations from Theorem \ref{mycoh} implies that the sum of $\phi(R^{m-|J|}V_J)$ taken over all subgees $J$ for which $J\cap[\![k]\!]=\{k\}$ must be 0. This sum includes the term $\phi(R^{m-1}V_k)=1$, while all other terms in the sum have $J\not\subset[\![k]\!]$ and hence have $\phi(R^{m-|J|}V_J)=0$, contradicting that the sum is 0.\end{proof}

The next two propositions are special cases of the theorem.  If $S\subset[\![t]\!]$, let $\widetilde S=[\![t]\!]-S$. We repeat that $\psi(R^{m-1-|I|}V_I)=\phi(R^{m-|I|}V_I)$ is assumed.
\begin{prop}\label{shortprop}  Let $T\subset[\![m+2]\!]$ with $|T|\le m$, and reindex as $T=[\![t+1]\!]$. Suppose $\phi(R^m)=1$, $\phi(R^{m-|T|}V_T)=1$, and $\phi(R^{m-|I|}V_I)=0$ for all $I\subsetneq T$ such that  $t+1\in I$. Then
$$(\phi\ot\psi)(\Vb{1}\cdots\Vb{t}\,\Vb{t+1}^{m-t}\Rbar^{m-1})=1.$$\end{prop}
\begin{proof} The expression expands as
$$\sum_{S\subset[\![t]\!]}\sum_{i=0}^{m-t}\tbinom{m-t}i\tbinom{m-1}{m-|S|-i}\phi(V_SV_{t+1}^iR^{m-|S|-i})\psi(V_{\widetilde S}V_{t+1}^{m-t-i}R^{|S|+i-1}).$$
The only terms with $\phi\cdot\psi\ne0$ are those with $(S,i)=([\![t]\!],m-t)$ or $(\emptyset,0)$, and the coefficients of these are 1 and 0, respectively.
\end{proof}

\begin{prop} \label{longprop} Let $T\subset[\![m+2]\!]$ with $|T|\le m$, and reindex as $T=[\![t+1]\!]$.  Suppose  $\phi(V_I)=1$ for  $I\subsetneq T$, and $\phi(V_{T})=0$. Let $m=2^e+\Delta$, $1\le\Delta<2^e$. Then
\begin{enumerate}
\item $(\phi \ot\psi)(\Vb{1}\cdots\Vb{t}\,\Vb{t+1}^A\Rbar^{2m-A-t-1})=\tbinom{2m-t}m+\tbinom{2m-A-t}{m-t}$.
\item $\tbinom{2m-t}m+\tbinom{2m-A-t}{m-t}$ is odd for the following values of $A$:
\begin{enumerate}
\item[a)] If $1\le t\le 2\Delta$, use $A=2\Delta+1-t$;
\item[b)] If $2\Delta+1\le t\le 2^e$ and $\binom{t-\Delta-1}\Delta$ is even, use $A=\Delta$.
\item[c)] If $2\Delta+1\le t\le 2^e$ and $\binom{t-\Delta-1}\Delta$ is odd, use $A=2^{\lg(2\Delta)}$.
\item[d)] If $2^e<t\le m-1$ and $\binom{m-t+\Delta}\Delta$ is even, use $A=\Delta$.
\item[e)] If $2^e<t\le m-1$ and $\binom{m-t+\Delta}\Delta$ is odd, use $A=\Delta-2^{\nu(m-t)}$.
\end{enumerate}
\end{enumerate}
\end{prop}
\begin{proof}
Part (1). The expression expands as
$$\sum_{S\subset[\![t]\!]}\sum_{i=0}^A\tbinom Ai\tbinom{2m-A-t-1}{m-i-|S|}\phi(V_SV_{t+1}^iR^{m-i-|S|})\psi(V_{\widetilde S}V_{t+1}^{A-i}R^{m-A-t-1+i+|S|}).$$
 If it were the case that all $\phi(-)=1$, this would equal $\sum\binom ts\binom{2m-t-1}{m-s}=\binom{2m-1}m\equiv0$, since we assume $m$ is not a 2-power. Our given values of $\phi$ differ from this only for $S=[\![t]\!]$ and $i>0$, or $S=\emptyset$ and $i<A$. Thus the sum becomes
\begin{eqnarray*}&&\sum_{i=1}^A\tbinom Ai\tbinom{2m-A-t-1}{m-i-t}+\sum_{i=0}^{A-1}\tbinom Ai\tbinom{2m-A-t-1}{m-i}\\
&=&\tbinom{2m-t-1}{m-t}+\tbinom{2m-A-t-1}{m-t}+\tbinom{2m-t-1}m+\tbinom{2m-A-t-1}{m-A}\\
&=&\tbinom{2m-t}m+\tbinom{2m-A-t}{m-t}.\end{eqnarray*}

Part 2a. The expression equals $\binom{2^{e+1}+2\Delta-t}{2^e+\Delta}+\tbinom{2^{e+1}-1}{2^e-\Delta+t-1}=0+1$.

Part 2b. The expression equals $\binom{2^{e+1}+2\Delta-t}{2^e+\Delta}+\binom{2^{e+1}+\Delta-t}{2^e}$. The first of these is congruent to $\binom{-(t-2\Delta)}\Delta\equiv\binom{t-\Delta-1}\Delta\equiv0$ by assumption, while the second is 1 since $t\le 2^e$.

Part 2c. The first term is as in (b), except now it is 1 by assumption. Let  $\Delta=2^v+\delta$ with $0\le\delta<2^v$. The second term becomes $\binom{2^{e+1}+2\delta-t}{2^e-2^v+\delta}$, which is 0 since the top has a 0 in the $2^{v+1}$ position, while the bottom has a 1 there.

Part 2d. The expression is the same as in part b. Now the first term is of the form $\binom{2^e+m+\Delta-t}{2^e+\Delta}\equiv0$ since $\binom{m+\Delta-t}{\Delta}\equiv0$ by assumption. The second term is 1 since its top is between $2^e$ and $2^e+\Delta$.

Part 2e. Let $m-t=2^w\a$ with $\a$ odd. The first term is $\binom{2^e+2^w\a+\Delta}{2^w\a}\equiv\binom{2^w\a+\Delta}{2^w\a}\equiv1$ by assumption. The second term is $\binom{2^e+2^w\a+2^w}{2^w\a}$, yielding 0 due to the $2^w$ position.

\end{proof}

We need one more lemma, in which $\cP(S)$ denotes the power set of $S$.
\begin{lem}\label{setlem} If $U$ is a set, and $\cC\subset\cP(U)$ with $\emptyset\in\cC$, then either
\begin{enumerate}
\item[a)] $\cC=\cP(X)$ for some $X\subset U$, or
\item[b)] there exists $T\subset U$ with $|T|\ge2$ such that $\cC\cap\cP(T)=\cP(T)-\{T\}$, or
\item[c)] there exists $s\in S\subset U$ such that $\{C\in\cC:\ s\in C\}\cap\cP(S)=\{S\}$, and $|S|\ge2$.
\end{enumerate}
\end{lem}
\begin{proof} Let $X=\{t\in U:\ \{t\}\in\cC\}$. If $\cP(X)\not\subset\cC$, then a minimal element $T$ of $\cP(X)-\cC$ is of type (b), and we are done.
If $\cC=\cP(X)$, we are done by (a).

 Thus we may assume that $\cP(X)\subsetneq\cC$. Choose $S'\in\cC-\cP(X)$ and $s\in S'-X$. Note that $\{s\}\not\in\cC$. Let $S$ be a minimal element in $\{C\in\cC:\ s\in\cC\}$.
 This $S$ is of type (c), so we are done.\end{proof}

\begin{proof}[Proof of Theorem \ref{nonzthm}] Let $U$ denote the set of $i$'s such that $V_i$ is a factor of some nonzero monomial in $H^m(\Mbar(\ell))$.  Let $\cC=\{I\subset U:\ \phi(R^{m-|I|}V_I)\ne0\}$. By Lemma \ref{ref}, either (b) or (c) of Lemma \ref{setlem} is true of $\cC$. If (b), then the theorem is true by Proposition \ref{longprop}, and if (c), then the theorem is true by Proposition \ref{shortprop}.

\end{proof}

\begin{proof}[Proof of Theorem \ref{str1}] By Theorem \ref{nonzthm}, (\ref{ineqs}) holds whenever $R^m\ne0$, and by Theorem \ref{genlthm}, it holds for $n=8$ (hence $m=5$) whenever $R^m=0$ and there is some nonzero monomial $R^{m-r}V_{i_1}\cdots V_{i_r}$ with $1\le r\le 3$. The only genetic code with $n=8$ having a gee of length 5 is that of the $5$-torus, $854321$. The majority of the genetic codes with $n=8$ do not have any gees of size 4, and  the theorem follows immediately for these. However, there are many genetic codes with $n=8$ having a 4-gee of the form 4321, 5321, 6321, 7321, 5421, 5431, or 5432, plus perhaps other, shorter, gees. This can be seen in \cite{web}, or deduced from the definitions. We will show that, except in the excluded cases,  if all monomials $M_S$ in $H^m(\Mbar(\ell))$ corresponding to subgees $S$ of size $\le3$ are 0, then so are the monomials in $H^m(-)$ corresponding to the 4-gees, contradicting that $H^m(-)\approx\zt$.

If the 4-gee is 4321 and there are any other gees, then 5 is a subgee, and $\cR_5$ (see Theorem \ref{mycoh}(3)) in $H^m(\Mbar(\ell))$ is a sum of the monomial $M_{4321}$ corresponding to 4321 plus monomials $M_S$ corresponding to certain subgees of size $\le3$. If all these $M_S=0$, then the relation implies that $M_{4321}=0$ and hence $H^m(-)=0$.

If the only gees and subgees of size 4 are one or more of 5321, 5421, 5431, or 5432, let $T$ denote any one of them, and let $j$ denote the element of $[\![5]\!]$ not contained in $T$.
The relation $\cR_j$ in $H^m(-)$ is a sum of the monomial $M_T$ corresponding to $T$ plus monomials $M_S$ corresponding to subgees $S$ of size $\le3$. If all $M_S$ are 0, then so is $M_T$, and since this holds for all $T$, we deduce $H^m(-)=0$.

If 6321 is the 4-gee and there is another, shorter, gee, then the shorter gee must contain 7. This can be seen at \cite{web}, or can be deduced as follows. [\![Any gee of size $\le 3$ which does not contain 7 and is not $\le6321$ would have to be $\ge54$. But if 86321 is short, then 754 is long and hence so is 854.]\!] Thus 7 is a subgee. The gees and subgees of size 4 are 6321, 5321, and 4321. If all monomials $M_S$ in $H^m(-)$ corresponding to subgees $S$ of size $\le3$ are 0, then the relations $\cR_7$, $\cR_6$, and $\cR_5$ imply, respectively,
$M_{6321}+M_{5321}+M_{4321}=0$, $M_{5321}+M_{4321}=0$, and $M_{6321}+M_{4321}=0$, and hence $H^m(-)=0$.

If 6321 is the only gee, then we do not have $\cR_7$ to work with, and in fact $M_{6321}=M_{5321}=M_{4321}\ne0$ with all monomials corresponding to shorter gees being 0. Hence in this case, we cannot use Theorem \ref{genlthm} to deduce (\ref{ineqs}) when $n=8$.

If 7321 is a gee, there can be no other gees, as can be seen from \cite{web} or deduced similarly to the deduction involving 6321. If all monomials $M_S$ corresponding to subgees of size $\le3$ are 0, then the relations $\cR_7$, $\cR_6$, $\cR_5$, and $\cR_4$ imply that $M_{7321}=M_{6321}=M_{5321}=M_{4321}=0$ and hence $H^m(-)=0$.
\end{proof}

\section{The effect of increasing $n$}\label{Incrsec}
In this section we discuss the effect of increasing $n$, while leaving the gees fixed, on the length vectors and the cohomology ring.

The operation of increasing the number of edges by 1 while leaving the gees unchanged has the following nice interpretation. Two length vectors  are said to be {\it equivalent} if they have the same genetic code, or equivalently their moduli spaces $\Mbar(\ell)$ are homeomorphic.
\begin{prop} Any generic length-$n$  vector is equivalent to one, $(\ell_1,\ldots,\ell_n)$ with $\ell_1\le\cdots\le\ell_n$, in which all lengths are positive integers and \begin{equation}\label{lcond}\ell_n+\ell_{n-1}\le\ell_1+\cdots+\ell_{n-2}+1.\end{equation}
If $\ell$ satisfies (\ref{lcond}), define a new vector $\ell'$ of length $n+1$ by
\begin{equation}\label{new}(\ell_1,\ldots,\ell_{n-1},\tfrac{|\ell|+1}2-\ell_n,\tfrac{|\ell|+1}2),\end{equation}
where $|\ell|=\ell_1+\cdots+\ell_n$. Then $\ell$ and $\ell'$ have the same gees.
\end{prop}
\begin{proof} It is shown in \cite{HR} that any length vector is equivalent to one, $\ell$, with nondecreasing positive integer entries. Since $\ell$ is generic, $|\ell|$ is odd. If it has $\ell_n+\ell_{n-1}=\ell_1+\cdots+\ell_{n-2}+2d+1$ for $d\ge 1$, we can find an equivalent length vector with smaller $d$, and hence eventually satisfy (\ref{lcond}),  as follows. (a) If $\ell_{n-1}>\ell_{n-2}$, then decrease $\ell_n$ and $\ell_{n-1}$ by 1. (b) If $\ell_n>\ell_{n-1}=\cdots=\ell_{n-t}>\ell_{n-t-1}$ for some $t\ge 2$, decrease $\ell_{n-1},\ldots,\ell_{n-t}$ by 1, and decrease $\ell_n$ by $t$. (c) If $\ell_n=\ell_{n-1}=\ell_{n-2}$, decrease each of them by 2.

We  explain briefly why each of these changes does not affect the genetic code. (a) The only short subsets containing $n$, either before or after the change, are of the form $T=S\cup\{ n\}$ with $S\subset[\![n-2]\!]$, and $\sum_{i\in T}\ell_i-\sum_{i\not\in T}\ell_i$ is invariant under the change. (b) First we show that the new vector has nondecreasing entries. We need $\ell_n-\ell_{n-1}\ge t-1$ for $t\ge3$, and this follows easily from $\ell_n+\ell_{n-1}\ge(t-1)\ell_{n-1}+3$. Because in either the old or new vector, the sum of the last two components is greater than half the total, the only short subsets containing $n$ will be of the form $T=S\cup\{n\}$ with $S\subset[\![n-t-1]\!]$, and the change subtracts $t$ from both $T$ and its complement. (c) The hypothesis implies that $\ds\sum_{j=1}^{n-3}\ell_j\le\ell_n-3$, and so the new vector has nondecreasing entries, and both the old and new vectors have $[\![n-3]\!]$ as the only gee.

Now we compare the gees of $\ell$ satisfying (\ref{lcond}) with those of (\ref{new}), which satisfies its own analogue of (\ref{lcond}), in fact with equality.

If equality holds in (\ref{lcond}), then for $\ell$ (resp.~$\ell'$), the only short subsets containing $n$ (resp.~$n+1$) are of the form $T=S\cup\{ n\}$ (resp.~$S\cup\{n+1\}$) with $S\subset[\![n-2]\!]$, and the change adds $\frac{|\ell|+1}2-\ell_n$ to both $\ds\sum_{j\in T}\ell_j$ and $\ds\sum_{j\not\in T}\ell_j$, thus leaving the gees unaffected.
If, on the other hand, $\ell_n+\ell_{n-1}\le\ell_1+\cdots+\ell_{n-2}-1$, then $\ell$ has some short subsets of the form $T_1=S_1\cup\{n-1,n\}$ with $S_1\subset[\![n-2]\!]$, and others of the form $T_2=S_2\cup\{n\}$ with $S_2\subset[\![n-2]\!]$. The vector $\ell'$ of (\ref{new}) will have exactly the sets $S_1\cup\{n-1,n+1\}$ and $S_2\cup\{n+1\}$ as its short subsets containing $n+1$, since again we add $\frac{|\ell|+1}2-\ell_n$ to both $\ds\sum_{j\in T_\eps}\ell_j$ and $\ds\sum_{j\not\in T_{\eps}}\ell_j$, $\eps=1,2$.
\end{proof}

Next we point out the effect on $H^*(-)$ of increasing $n$ while leaving the gees unchanged.
Recall that $m=n-3$. If $s$ denotes the size of the largest gee and $m\ge 2s$, then for $d\le m-s$, $H^d(\Mbar(\ell))$ has a basis $$\{R^{d-|S|}V_S: S\text{ is a subgee with }|S|\le d\},$$ so in this range increasing $m$ leaves the cohomology groups unchanged, while for $m-s<d\le m$, the relations among the subgees depend on $m-d$, and so in this range increasing $m$ acts like suspending.

For our strong TC results, we only use $H^{m-1}(-)$ and $H^m(-)$. We illustrate with a simple example: monogenic codes 7321 and 8321, so the gee is 321 and $m=4$ and 5.  One can verify that in this case $H^m(-)\approx\la R^{m-3}V_1V_2V_3\ra$ and
$$H^{m-1}(-)\approx\la R^{m-3}V_1V_2, \  R^{m-3}V_1V_3,\ R^{m-3}V_2V_3,\ R^{m-4}V_1V_2V_3\ra.$$
So $r=3$ in Theorem \ref{genlthm}, and the theorem applies for $m\ge 5$. Note that in Theorem \ref{genlthm}, we are only considering homomorphisms $\psi:H^{m-1}(-)\to\zt$ which are essentially equal to $\phi:H^m(-)\to\zt$, so that in this case the theorem is only utilizing the class $R^{m-4}V_1V_2V_3$ in $H^{m-1}(-)$.

When $m=5$, $(\phi\ot\psi)({\Vb1}^3\,{\Vb2}^3\,{\Vb3}^3)\ne0$ in bidegree $(5,4)$ as it equals
$$\phi(V_1V_2^2V_3^2)\psi(V_1^2V_2V_3)+\phi(V_1^2V_2V_3^2)\psi(V_1V_2^2V_3)+\phi(V_1^2V_2^2V_3)\psi(V_1V_2V_3^2)\ne0.$$
(Keep in mind part (1) of Theorem \ref{mycoh}.)

When $m=4$, we would need
\begin{equation}\label{71}{\Vb1}^i\,{\Vb2}^j\,{\Vb3}^k\,\Rbar^{7-i-j-k}\ne0\end{equation}
in bidegree $(4,3)$. The only way to get $RV_1V_2V_3\ot V_1V_2V_3$ would be with none of $i$, $j$, or $k$ being a 2-power, and this is impossible with $i+j+k\le7$. We could get $RV_1V_2V_3\ot RV_1V_2$ from $i=j=3$, $k=1$, but this would have even coefficient from
$$V_1^2V_2V_3\ot V_1V_2^2+V_1V_2^2V_3\ot V_1^2V_2.$$
We conclude that (\ref{71}) is impossible, and so we cannot deduce $\TC(\Mbar(\ell))\ge8$ for this $\ell$ when $m=4$ by cohomological considerations.
Thus Theorem \ref{genlthm} is optimal for the gee 321 in obtaining the strong lower bound for its TC when $m\ge5$.

If the only gee is $[\![r]\!]$, we easily see, using Theorem \ref{mycoh}, that the only nonzero monomial in
$H^m(\Mbar(\ell))$ is $R^{m-r}V_{[\![r]\!]}$, while those in $H^{m-1}(\Mbar(\ell))$ are $R^{m-1-r}V_{[\![r]\!]}$ and $R^{m-2-r}V_{[\![r]\!]-\{i\}}$ for $1\le i\le r$. We can sometimes improve upon Theorem \ref{genlthm} because of this explicit information about $H^{m-1}(-)$.

For example, if the only gee is $[\![4]\!]$, Theorem \ref{genlthm} said $\TC\ge 2m$ for $m\ge8$, but we can also deduce $\TC\ge 2m$ for $m=7$ and $m=6$ using
\begin{equation}\label{7}\Vb{1}^3\,\Vb{2}^3\,\Vb{3}^3\,\Vb{4}^4=V_1V_2V_3V_4^4\ot V_1^2V_2^2V_3^2\text{ in bidegree }(7,6)\end{equation}
and
$$\Vb{1}^3\,\Vb{2}^3\,\Vb{3}^3\,\Vb{4}^2=V_1^2V_2V_3V_4^2\ot V_1V_2^2V_3^2\text{ in bidegree }(6,5).$$
When $m=5$, we cannot deduce $\TC\ge10$ using zcl. When $m=4$, so the genetic code is 74321, $\Mbar(\ell)$ is homeomorphic to a 4-torus with $\TC=5$.

If the only gee is $[\![r]\!]$ for $r=5$, 6, or 7, Theorem \ref{genlthm} says $\TC\ge 2m$ for $m\ge r+4$, and one can check that this is all we can do, in the sense that, for $m<r+4$, $\Vb{1}^{i_1}\cdots\Vb{r}^{i_r}\,\Rbar^{2m-1-i_1-\cdots-i_r}=0$.

\section{Specific results for monogenic codes of arbitrary length}\label{sec2}
In this section, we discuss some families of monogenic genetic codes of arbitrary length in which we can show that $R^m=0$ and find the value of $r$ that works in Theorem \ref{genlthm}.
At the end, we discuss evidence suggesting that it is quite rare for a monogenic code to  have $\zcl< 2n-7$ (and hence not be able to deduce $\TC\ge 2n-6$ from zcl).

\begin{defin}Let $\cS_k$ denote the set of $k$-tuples of nonnegative integers such that, for all $j$, the sum of the first $j$ components is $\le j$.
\end{defin}
For $B=(b_1,\ldots,b_k)$, let $|B|=\sum b_i$. The following theorem is the main result of \cite{D}.
\begin{thm}\label{Dthm} Suppose $\ell$ has a single gee, $\{g_1,\ldots,g_k\}$, with $a_i=g_i-g_{i+1}>0$. ($a_k=g_k$.) If $J$ is a set of distinct integers $\le g_1$, let $\theta(J)=(\theta_1,\ldots,\theta_k)$, where $\theta_i$ is the number of elements  $j\in J$ satisfying
$g_{i+1}<j\le g_i$.
Then, if $\ell$ has length $m+3$, the Poincar\'e duality isomorphism $\phi:H^m(\Mbar(\ell))\to\zt$ satisfies
\begin{equation}\label{phitop}\phi(R^{m-r}V_{j_1}\cdots V_{j_r})=\sum_B\prod_{i=1}^k\tbinom{a_i+b_i-2}{b_i},\end{equation}
where $B$ ranges over all $(b_1,\ldots,b_k)$ for which $|B|=k-r$ and $B+\theta(\{j_1,\ldots,j_r\})\in\cS_k$.\end{thm}

The following corollary is useful.
\begin{cor} In  Theorem \ref{Dthm}, (\ref{phitop}) depends only on the reductions $a_i\mod 2^{\lg(2i)}$.
\end{cor}
\begin{proof} For a $B$-summand in (\ref{phitop}) to be nonzero, it is necessary that each $b_i$ be $\le i$. Binomial coefficients $\binom xb$  depend only on $x\mod2^{\lg(2b)}$.\end{proof}

We often write $\abar_i$ for the mod $2^{\lg(2i)}$-reduction of $a_i$. In Theorems \ref{thm01} and \ref{thm02}, we describe two infinite families of $(\abar_1,\ldots,\abar_k)$ for which Theorem \ref{genlthm} applies, and so we can deduce the strong lower bound for $\TC(\Mbar(\ell))$.

Let $w_i=V_{g_i}$ (or any $V_j$ satisfying $g_{i+1}<j\le g_i$).
 If $I=(\eps_1,\ldots,\eps_k)$ with each $\eps_j\in\{0,1\}$, we  let
\begin{equation}\label{YI}Y_I=R^tw_1^{\eps_1}\cdots w_k^{\eps_k}\end{equation} for any $t$. This notation for $Y_I$ could be extended to include  products of distinct $V_j$ with $j$-values in the same subinterval $(g_{i+1},g_i]$, but the  consideration of such $Y_I$ seems not to be useful. The total grading of our classes $Y_I$ will be implicit, usually $m$ or $m-1$, and the value of $t$ in (\ref{YI}) is chosen to make the class have the desired grading.
 Then (\ref{phitop}) could be restated as
 \begin{equation}\label{Dth}\phi(Y_I)=\sum_B\prod_{i=1}^k\tbinom{a_i+b_i-2}{b_i}\end{equation}
in grading $m$, where $B$ ranges over all $(b_1,\ldots,b_k)$ such that $B+I\in\cS_k$ and $|B+I|=k$.

For monogenic codes as in Theorem \ref{Dthm}, we can easily check whether $R^m=0$, (or equivalently $\phi(R^m)=0$), since $R^m$ is just $Y_I$ with $I$ consisting of all 0's. Thus $\phi(R^m)=\sum_B\tbinom{a_i+b_i-2}{b_i}$ with the sum taken over all $B=(b_1,\ldots,b_k)$ satisfying $|B|=k$ and $B\in\cS_k$. For example, when $k=3$
$$\phi(R^m)= \tbinom{a_3+1}2+\tbinom{a_2}2a_3'+\tbinom {a_3}2(a_1'+a_2')+a_1'a_2'a_3',$$
where $a_i'=a_i-1$.

For a given $k$, there are  $\ds\prod_{i=1}^k2^{\lg(2i)}$ possibilities for $(\abar_i,\ldots,\abar_k)$.
 {\tt Maple} determined the information in Table \ref{T1} regarding how many of these have $R^m=0$.

 \begin{tab}\label{T1}

\begin{center}
{\scalefont{1}{
$$\renewcommand\arraystretch{1.0}\begin{array}{c|rr}
k&\#\  \abar\text{'s}&\# \ R^m=0\\
\hline
3&32&20\\
4&256&128\\
5&2048&1216\\
6&16384&9600
\end{array}$$}}
\end{center}
\end{tab}

Our simplest result follows.  Keep in mind that in the rest of this section $\abar_i$ refers to the reduction mod $2^{\lg(2i)}$ of $a_i=g_i-g_{i+1}$, where $\ell$ has a single gene $\{m+3,g_1,\ldots,g_k\}$.

\begin{thm}\label{thm01} If there is a set $Z\subset[\![k]\!]$ with $|Z|=r$ such that
$$\abar_i= \begin{cases}1&i\in Z\\
0&i\not\in Z,\end{cases}$$
then the hypothesis of Theorem \ref{genlthm} holds for $R^{m-r}\ds\prod_{i\in Z}w_i$, and hence
$\TC(\Mbar(\ell))\ge 2m$ if $m\ge r+2^{\lg(r)}$.
\end{thm}
\begin{proof} For $ \abar_i\in\{0,1\}$, the only times that $\binom{a_i+b_i-2}{b_i}$ is odd are when $b_i=0$ or ($b_i=1$ and $\abar_i=0$). Thus a nonzero term in (\ref{Dth}) requires $|B|\le k-r$. For $|I|\le r$, $\phi(Y_I)$ can have nonzero terms only if $|B|=k-r$ and $|I|=r$, and moreover only for the single $B$ given by
$$b_i=\begin{cases}1&i\not\in Z\\
0&i\in Z.\end{cases}$$
If $I$ has 1's in the positions in $Z$, and 0's elsewhere, $\phi(Y_I)=1$ from the single term with $B$ as above and $B+I=(1,\ldots,1)$.
\end{proof}

For example, the case $(\abar_1,\abar_2,\abar_3,\abar_4)=(1,1,0,1)$ applies to any $\ell$ whose genetic code is $\la\{n,g_1,g_2,g_3,g_4\}\ra$ with $n>g_1>g_2>g_3>g_4$ and $g_4\equiv1\ (8)$, $g_3\equiv1\ (4)$, $g_2\equiv2\ (4)$, and $g_1\equiv1\ (2)$. Theorem \ref{thm01} says that $\TC(\Mbar(\ell))\ge 2n-6$ if $n-3\ge 3+2$, so if $n\ge8$. But the smallest values that satisfy the conditions on $n$ and $g_i$ are 87651, and so the condition in the theorem that $m\ge r+2^{\lg r}$ covers all possibilities here. Moreover, there is no $\ell$ having this as genetic code because 432 and its complement would both be short. This suggests that the condition in Theorem \ref{genlthm} that $m\ge r+2^{\lg r}$ will usually cover all possible values of $m$, and so the ``sufficiently large'' in our title will rarely need to be invoked. We will discuss this in more detail at the end of this section.

A result similar to Theorem \ref{thm01} holds if some of the 0's in $\la \abar_i\ra$ are immediately followed by a 2.
\begin{thm}\label{thm02} If there are disjoint subsets  $T,Z\subset[\![k]\!]$ such that
$$\abar_i=\begin{cases}2&i\in T\\
1&i\in Z\\
0&i\in C:=[\![k]\!]-(T\cup Z)\end{cases}$$
and $\abar_{i-1}=0$ whenever $i\in T$, then, with $r=|T|+|Z|$, the hypothesis of Theorem \ref{genlthm} holds for $R^{m-r}\ds\prod_{i\in  Z}w_i\cdot\prod_{i\in T}w_{i-1}$,
and hence
$\TC(\Mbar(\ell))\ge 2m$ if $m\ge r+2^{\lg(r)}$.
\end{thm}
\begin{proof} Let $T'=\{t-1:\ t\in T\}$. Then $T'\subset C$. Let $C'=C-T'$.

Let $I=(\eps_1,\ldots,\eps_k)$ with $\eps_j\in\{0,1\}$ be given.
We will show that $\phi(Y_I)=0$ if $|I|<r$, and $\phi(Y_I)=1$ if $\eps_i=1$ if $i\in Z\cup T'$, and $\eps_i=0$ otherwise.

If $U$ is a $k$-tuple of nonnegative integers, define
$$\chi(U)=\begin{cases}1&\text{if }U\in\cS_k\\
0&\text{if }U\not\in\cS_k.\end{cases}$$
Let $\cB$ denote the set of $k$-tuples $B=(b_1,\ldots,b_k)$ such that $|B+I|=k$ and
$$b_i\begin{cases}
\in\N&i\in T\\
=0&i\in Z\\
\in\{0,1\}&i\in C.\end{cases}$$
Here $\N$ denotes the set of nonnegative integers.
These $b_i$ are the values for which $\binom{a_i+b_i-2}{b_i}\equiv1\mod 2$. By (\ref{Dth})
$$\phi(Y_I)=\sum_{B\in\cB}\chi(B+I)\in\zt.$$

Let
$$\cP=\{B\in\cB:\ b_{t-1}+b_t\ge2\text{ for some }t\in T\}.$$
For an element $B$ of $\cP$, choose the minimal $t$ with $b_{t-1}+b_t\ge 2$, and pair $B$ with the element $B'$ which has $b'_{t-1}=1-b_{t-1}$, $b'_t=b_t+2b_{t-1}-1$, with other entries equal to those of $B$. We show in the next paragraph that $\chi(B+I)=\chi(B'+I)$ for all $B\in\cP$. Thus
$$\sum_{B\in\cP}\chi(B+I)\equiv0\mod 2.$$

Assume, without loss of generality, that $b_{t-1}=0$. The only value of $h$ for which
$$\sum_{i=1}^h(B+I)_i\ne\sum_{i=1}^h(B'+I)_i$$
is $h=t-1$. The only conceivable way to have $\chi(B+I)\ne\chi(B'+I)$ is if
$$\sum_{i=1}^{t-2}(B+I)_i=t-2$$
and $\eps_{t-1}=1$, for then, since $b'_{t-1}=1$, $B'+I$ fails the condition to be in $\cS_k$ at position $t-1$, but $B+I$ doesn't. However, since $b_t\ge2$, $B+I$ fails at position $t$. Thus $\chi(B+I)=\chi(B'+I)$.

Now let $\cQ=\cB-\cP$. If $B\in\cQ$, then
$$|B|=\sum_{T'\cup T}b_i+\sum_{C'}b_i+\sum_Zb_i\le|T|+|C'|+0=|C|=k-r.$$
Since $|I|\le r$, we must have $|B|=k-r$ and $|I|=r$ in order to get a nonzero term in (\ref{Dth}). Thus $\phi(Y_I)=0$ whenever $|I|<r$.

The $Y_I$ being considered has $\eps_i=1$ for $i\in Z\cup T'$. The $B$'s in $\cQ$ which might give a nonzero term in (\ref{Dth}) must have $(b_{t-1},b_t)=(0,1)$ or $(1,0)$ for all $t\in T$, and $b_i=1$ if $i\in C'$. Thus $B+I$ has 1's in $Z\cup C'$, and $(1,1)$ or $(2,0)$ in each position pair $(t-1,t)$. If any $(2,0)$ occurs, then $B+I$ fails to be in $\cS_k$, and so (\ref{Dth}) has only one nonzero term, namely where $(b_{t-1},b_t)=(0,1)$ for all $t\in T$. Thus $B+I\in\cS_k$, and $\phi(Y_I)=1$.\end{proof}

A similar result and proof holds when $\la\abar_i\ra$ contains subsequences of the form $(0,0,3)$ or $(0,1,2)$. Moreover, all these can be combined, so that  (\ref{ineqs}) holds
 for $m\ge r+2^{\lg(r)}$, where $r$ is the number of 1's plus the number of (0,2), (0,0,3), and (0,1,2) sequences in $\la\abar_i\ra$, provided that $\la\abar_i\ra$ contains only 0's and 1's and these sequences. However, there are many other sequences $\la\abar_i\ra$ not of this type for which the result holds.

In the remainder of the paper, we present more evidence that it is very rare that we cannot prove (\ref{ineqs}). For genetic codes with a single gene of size 4, it was shown in \cite{Forum} that $\zcl(\Mbar(\ell))\ge 2n-7$  unless the gene is 6321, 7321, or 7521. Here we prove a totally analogous result for genes of size 5.
\begin{thm} \label{str2} For a genetic code with a single gene of size 5, $\zcl(\Mbar(\ell))\ge 2n-7$  except when the gene is $74321$, $84321$, or $86321$.\end{thm}

The proof will use the following general lemma. Notation is as in (\ref{YI}).
\begin{lem}\label{toplem} For a monogenic code with gene $\{m+3,g_1,\ldots,g_k\}$ and $a_i=g_i-g_{i+1}>0$, all monomials $R^{m-t}w_{i_1}\cdots w_{i_t}$ for $t<k$ are 0 if and only if $a_i\equiv1\mod 2^{\lg(2i)}$ for all $i$.\end{lem}
\begin{proof} Let $\Yhat_J$ denote $Y_I$ where $I$ has 0's in the positions in $J$ and 1's elsewhere. Let $a_i'=a_i-1$. Using (\ref{Dth}), $\phi(\Yhat_j)=a_j'+\cdots+a_k'$, and hence all $\phi(\Yhat_j)$ are 0 iff $a_i$ is odd for all $i$. Similarly, $\phi(\Yhat_{j_1,j_2})$ with $j_1<j_2$ has terms with factors $a_i'$, which are 0 by the above, plus terms in (\ref{Dth}) with $B=2\varepsilon_j$ for $j\ge j_2$. For $\phi(\Yhat_{j_1,j_2})$ to be 0, we must have $\binom{a_{j_2}}2+\cdots+\binom{a_k}2=0$. For all these to be 0, we must now have $a_i\equiv1$ mod 4 for all $i>1$.
Similarly, $\phi(\Yhat_{j_1,j_2,j_3,j_4})$ will be a sum of terms already shown to be 0 plus $\binom{a_{j_4}}4+\cdots+\binom{a_k}4$. For all these to be 0, we must have $a_i\equiv 1$ mod 8 for $i\ge4$. Continuing in this way implies the result.\end{proof}

\begin{proof}[Proof of Theorem \ref{str2}] The theorem holds for those gees having $R^m\ne0$ by Theorem \ref{nonzthm}. For those with $R^m=0$ and $(\abar_1,\abar_2,\abar_3,\abar_4)\ne(1,1,1,1)$, the result holds for  $m\ge5$ (hence for $n\ge8$) by Lemma \ref{toplem} and Theorem \ref{genlthm}. But the only  length-5 gene with $n<8$ is 74321.

It remains to consider codes with $(\abar_1,\abar_2,\abar_3,\abar_4)=(1,1,1,1)$. Theorem \ref{genlthm} implies the result for $m\ge8$, so for $n\ge11$. We must consider the gees with $g_1\le 9$ (and all $\abar_i=1$). These are 4321, 6321, 8321, 8721, and 8761. As noted at the end of the previous section, for the gee 4321, we have
\begin{equation}\label{1111} H^{m-1}(\Mbar(\ell))=\la R^{m-5}V_1V_2V_3V_4,R^{m-4}V_1V_2V_3,R^{m-4}V_1V_2V_4,R^{m-4}V_1V_3V_4,R^{m-4}V_2V_3V_4\ra\end{equation}
and $\zcl(\Mbar(\ell))\ge 2n-7$ for $n\ge9$ but not for $n=7$ or 8.

We now study the case when the gee is 6321. From (\ref{1111}), we deduce that there is a uniform homomorphism $\psi:H^{m-1}(\Mbar(\ell))\to\zt$ sending only $R^{m-4}V_1V_2V_4$, $R^{m-4}V_1V_2V_5$, and $R^{m-4}V_1V_2V_6$ nontrivially. ``Uniform,'' as discussed in \cite{Forum}, means here that $\psi$ treats $V_4$, $V_5$, and $V_6$ identically because of the interval from 4 to 6, inclusive, in the gee. This dependence of uniform homomorphisms only on $\abar_i$ was observed in \cite{Forum}, but can be seen in this case directly as follows. [\![The only relations $\cR_J$ corresponding to gees of size $\ge2$ which involve any of these three monomials are $\cR_{3,j}$ for $j\in\{4,5,6\}$. Each of these relations involves only two of the three monomials and so would be sent to 0 by our $\psi$.]\!] Using this $\psi$, we have for $m=7$, similarly to (\ref{7}),
$$(\phi\ot\psi)(\Vb{1}^3\,\Vb{2}^3\,\Vb{3}^4\,\Vb{6}^3)=\phi(V_1V_2V_3^4V_6)\psi(V_1^2V_2^2V_6^2)\ne0,$$
and for $m=6$,
$$(\phi\ot\psi)(\Vb{1}^3\,\Vb{2}^3\,\Vb{3}^2\,\Vb{6}^3)=\phi(V_1^2V_2V_3^2V_6)\psi(V_1V_2^2V_6^2)\ne0.$$
One easily checks that nothing works when $m=5$. This establishes the claim when the gee is 6321.

A similar analysis works for gees 8321, 8721, and 8761. For each, there is a uniform homomorphism $\psi:H^{m-1}(\Mbar(\ell))\to\zt$ sending five classes $R^{m-4}V_1V_jV_k$ nontrivially, where $k$ ranges over the gap in the gee (for example 3, 4, 5, 6, and 7 in the second), and $j=2$, 2, and 7 in the three cases. There are products similar to those of the previous paragraph mapped nontrivially by $\phi\ot\psi$ when $m=7$ or 6. For example, for 8721 and $m=7$, $$(\phi\ot\psi)(\Vb{1}^3\,\Vb{2}^3\,\Vb{7}^3\,\Vb{8}^4)=\phi(V_1V_2V_7V_8^4)\psi(V_1^2V_2^2V_7^2)\ne0.$$
Thus $\zcl\ge 2m-1$ for these when $m=7$ or 6.
\end{proof}

We have performed a similar analysis for single genes of size 6, and found that again the only exceptions to $\zcl\ge 2n-7$ occur when all $\abar_i=1$. However, this time there are twelve such exceptional genes: (using $T$ for 10, and $E$ for 11), 854321, 954321, $T$54321, $E$54321, 974321, $T$74321, $E$74321, $T$94321, $E$94321, $T$98321, $E$98321, $E$98721. Details are available from the author upon request.

 \def\line{\rule{.6in}{.6pt}}

\end{document}